\documentclass[12pt, a4paper, oneside]{amsart}

\usepackage{packages}
\usepackage{xpatch}
\NewCommandCopy{\notocsection}{\section}
\xpatchcmd{\notocsection}{{1}}{{1001}}{}{}
\NewCommandCopy{\notocsubsection}{\subsection}
\xpatchcmd{\notocsubsection}{{2}}{{1001}}{}{}

\title{Folding and Metric Entropies for Extended Shifts}

\author[N. Martins]{Neemias Martins}
\address[NM]{Department of Mathematics, Unicamp, IMECC\\ Campinas-SP, 13083-970, Brazil}
\email{neemias@ime.unicamp.br}

\author[P. Mattos]{Pedro G. Mattos}
\address[PM]{Department of Mathematics, Unicamp, IMECC\\ Campinas-SP, 13083-970, Brazil}
\email{pedrogmattos@ime.unicamp.br}

\author[R. Varão]{Régis Varão}
\address[RV]{Department of Mathematics, Unicamp, IMECC\\ Campinas-SP, 13083-970, Brazil}
\email{varao@unicamp.br}



\begin{document}

\begin{abstract}
In this paper we calculate the metric and folding entropies for a family of non-invertible symbolic dynamical systems $(\Shift_{m_-,m_+}, \shift_\phi)$ which generalizes the standard bilateral Bernoulli shifts. The space $\Shift_{m_-,m_+}$ consists of symbolic sequences over two distinct finite alphabets, with dynamics governed by a shift map $\shift_\phi$ incorporating a non-invertible function $\phi$ that maps one of the alphabets to the other one. These systems are, for instance, particularly useful for encoding the many-to-one baker's transformation endomorphisms, and they can also be seen as a skew product with a unilateral Bernoulli shift on the base. 
\end{abstract}

\maketitle
\vspace{-\baselineskip}
\tableofcontents

\section{Introduction}

\subsection{Classical theory of Bernoulli shifts}

Bernoulli shifts play a central role in the study of dynamical systems, serving as a prototype for understanding chaotic and stochastic behavior.
Since \citeauthor{Hadamard1898}'s study of the geodesic flow on negatively curved manifolds \cite{Hadamard1898}, their significance is underscored by their connections to other fundamental results in the field.

For instance, Anosov flows and geodesic flows on manifolds with negative curvature exhibit strong mixing properties and are often modeled using Bernoulli shifts due to their similar chaotic characteristics \cite{Anosov1967,Zemer-Kosloff2024,ornstein1973geodesic,ratner1974anosov}. Similarly, hyperbolic toral automorphisms, such as Arnold's cat map, share ergodicity and mixing properties with Bernoulli shifts, making them discrete-time analogues in the study of hyperbolic dynamics \cite{Katznelson1971, Sinai1966}. More generally, \citeauthor{Bowen1970} used symbolic dynamics to study Axiom A diffeomorphisms \cite{Bowen1970}.

The profound impact of Bernoulli shifts is further highlighted by Ornstein's classification theorem \cite{Ornstein}, which established that metric entropy, as defined by Kolmogorov and Sinai \cite{Kolmogorov1958,Sinai1959}, is a complete invariant for isomorphisms of bilateral Bernoulli shifts, a result that has influenced the classification of a wide range of dynamical systems. These connections demonstrate the unifying role of Bernoulli shifts in linking geometric, topological, and measure-theoretic aspects of dynamics.

\subsection{The case of endomorphisms}

Although the classification of invertible dynamical systems, such as Bernoulli shifts, has seen significant progress --- most notably through Ornstein's isomorphism theorem --- the classification of non-invertible maps (endomorphisms) presents a far more complex and nuanced challenge. Unlike invertible systems, where entropy serves as a complete invariant for Bernoulli shifts, non-invertible systems lack a similarly comprehensive classification framework. Entropy alone is insufficient to classify one-sided Bernoulli shifts \cite[Section 4.5]{Downarowicz}.

An endomorphism has a Bernoulli extension if its natural extension in the limit space is isomorphic to a Bernoulli shift. However, two  non-isomorphic endomorphisms can have isomorphic natural extensions. Some characterizations and applications of non-invertible Bernoulli transformations are presented by \citeauthor{Ashley1997}~\cite{Ashley1997}, \citeauthor{Bruin2009}~\cite{Bruin2009}, and \citeauthor{Hoffman-Rudolph2002.1}~\cite{Hoffman-Rudolph2002.2, Hoffman-Rudolph2002.1}. 

Anosov endomorphisms are also a rich family of examples to once more illustrate how non-invertibility provides new behavior if compared to their invertible analogue. It is well known that Anosov diffeomorphisms are structurally stable, but this fails for Anosov endomorphism, even though from the inverse limit perspective they are structurally stable. In fact, not even a linear Anosov endomorphism is structurally stable~\cite{Feliks}.

The mathematical physicist David Ruelle introduced the folding entropy as a tool to measure the complexity of the preimage branches of an endomorphism~\cite{Ruelle}. Ruelle was interested in investigating entropy production in non-equilibrium statistical mechanics. Considering any state $\mu$ as a limit of a sequence of absolutely continuous measures with probability densities, Ruelle expressed the entropy production by a limiting process from which the folding entropy emerged as a conditional entropy measuring the complexities of preimages of the system in terms of the state $\mu$. Some recent applications of the folding entropy are given by \citeauthor{liao-wang}~\cite{liao-wang}, and \citeauthor{Wu-Zhu}~\cite{Wu-Zhu, wu2022}.

As we will properly present in the next subsection, our work deals with a certain case of endomorphism for which our two main results, \cref{thm:measure_entropy_of_extended_shift,thm:folding_entropy_formula}, are the calculation of the metric and folding entropy, respectively. Our system could also be seen as a skew product with a one-sided Bernoulli system on the base, but we do not want to take this approach and we deliberately see it as an extended shift, as it will become clear soon since we work with two alphabets.

\subsection{Shifts and extended shifts}

Before we present our systems, we establish some notation for standard unilateral and bilateral shifts.
The symbolic space in $m \in \Z_{> 0}$ symbols can be either $\Shift^+_m := \{0, \ldots, m-1\}^{\Z_{\geq 0}}$, the space of all \emph{unilateral sequences} $x = (x_0, x_1, \ldots)$, or $\Shift_m := \{0, \ldots, m-1\}^{\Z}$, the space of all \emph{bilateral sequences} $x = (\ldots, x_{-1}; x_0, x_1, \ldots)$. In both cases, the entries of the sequences are elements of the alphabet $\{0, \ldots, m-1\}$ (we will only consider finite alphabets). We also denote $\Shift^-_{m} := \{0, \ldots, m-1\}^{\Z_{<0}}$, the space of unilateral sequences $(\ldots, x_{-2}, x_{-1})$ on $m$ symbols indexed by the negative integers (this is not standard notation but we will use this ahead).
For unilateral sequences, the shift map $\fun{\shift}{\Shift^+_m}{\Shift^+_m}$ is defined by $\shift(x) := (x_1, x_2, \ldots)$, whereas for bilateral sequences we define $\fun{\shift}{\Shift_m}{\Shift_m}$ by $\shift(x) := (\ldots, x_{0}; x_1, x_2 \ldots)$. This map captures the behavior of the orbits of points in the system being encoded.

We will define an extension of bilateral Bernoulli shifts which can also be interpreted as a skew product of unilateral shifts with finite alphabets (see \cref{sec:extended_shifts} for precise definitions). Given two strictly positive integers $m_-,  m_+ \in \Z_{> 0}$, we define the product space $\Shift_{m_-,m_+} := \Shift^-_{m_-} \times \Shift^+_{m_+}$. The elements of $\Shift_{m_-,m_+}$ are pairs $x = (x^-, x^+)$, with $x^-=(\ldots, x_{-2}, x_{-1})$ and $x^+=(x_0, x_1,\ldots)$, which can be seen as bilateral sequences
    \begin{equation*}
    x = (\ldots, x_{-1}; x_0, x_1, \ldots).
    \end{equation*}

It is important to notice that the symbols indexed by strictly negative numbers (entries of $x^- \in \Shift^-_{m_-}$) are elements of the \emph{backward alphabet} $\{0, \ldots, m_- -1\}$, while the symbols indexed by positive numbers are elements of the  \emph{forward alphabet} $\{0, \ldots, m_+ -1\}$. This means that, when $m_- < m_+$, we cannot define the usual shift map on these sequences. To overcome this problem, we use a surjective function $\fun{\phi}{\{0, \ldots, 
m_+ -1\}}{\{0, \ldots, m_- -1\}}$ and define an alternative ``shift'' map $\fun{\shift_\phi}{\Shift_{m_-,m_+}}{\Shift_{m_-,m_+}}$ which shifts the sequence $x$ to the left, as usual, but translates the symbol $x_0$ from the alphabet $\{0, \ldots, m_+ -1\}$ to the symbol $\phi(x_0)$ in $\{0, \ldots, m_--1\}$; that is, $\shift_\phi(x)$ is given by
    \begin{equation*}
    \shift_\phi(\ldots, x_{-2}, x_{-1}; x_0, x_1, \ldots) := (\ldots, x_{-1}, \phi(x_0); x_1, x_2 \ldots)
    \end{equation*}
We call $(\Shift_{m_-,m_+}, \shift_\phi)$ the \emph{$(m_-, m_+)$-extended shift} (with \emph{transition function} $\phi$). When $m_+ = m_-$, this is isomorphic to $\Shift_{m_+} = \{0, \ldots, m_+ -1\}^{\Z}$ with the standard bilateral shift $\shift$, since $\phi$ will just be a permutation of the symbols.

This system has the structure of an iterated function system (IFS) given by a skew product of base $\Shift^+_{m_+}$ and fiber $\Shift^-_{m_-}$. (Notice that this inverts the usual order of the product in the notation for skew products, since we denote the fiber to the left and the base to the right; we do so because our perspective is primarily that of $\Shift_{m_-,m_+}$ as a space of \textit{bilateral} sequences.) If for each symbol $s \in \{0, \ldots, m_--1\}$ we define the map $\fun{\shift\inv_s}{\Shift^-_{m_-}}{\Shift^-_{m_-}}$, $\shift\inv_s(x^-) = (\ldots, x_{-1}, s)$, then the action of $\shift_\phi$ on $x = (x^-, x^+) \in \Shift_{m_-,m_+}$ is
    \begin{equation*}
    \shift_\phi(x^-, x^+) = (\shift\inv_{\phi(x_0)}(x^-), \shift(x^+)).
    \end{equation*}
The maps $\fun{\shift\inv_s}{\Shift^-_{m_-}}{\Shift^-_{m_-}}$ are easily shown to be contractions.

Apart from trivial cases, these systems are non-invertible. They can be used to encode non-invertible dynamics, which is not possible with the standard bilateral shift because it is invertible.
An example is the $m$-to-$1$ baker's transformation, which represents a non-invertible model of deterministic chaos and is a measure-preserving generalization of the classical ($1$-to-$1$) baker’s transformation.
They are also mixing, ergodic and have chaotic behavior: they are transitive maps with dense periodic points \cite{Mehdipour-Martins}.

\subsection{Main results}

Given a probability distribution $p^+ = (p^+_0, \ldots, p^+_{m_+ -1})$ on the forward alphabet $\{0, \ldots, m_+ -1\}$, we define a measure $\med$ on $\Shift_{m_-,m_+}$ with respect to which the extended shift $\shift_\phi$ is measure-preserving and ergodic (see \cref{ssec:measure_structure} for details). The first main result we obtain relates the Kolmogorov-Sinai metric entropy $h_\med(\shift_\phi)$ of the extended shift $\shift_\phi$ with the metric entropy $\Enm(\mathcal C_0)$ of the partition $\mathcal C_0$ by cylinders $\Cyl{0}{s} = \set{x \in \Shift_{m_-, m_+}}{x_0 = s}$, $0 \leq s < m_+$ (check \cref{ssec:measurable_structure}), which is just the Shannon entropy of $p^+$.
In particular, it shows $\enm(\shift_\phi)$ is not dependent on the transition function $\phi$.

\begin{introtheorem}
\label{thm:measure_entropy_of_extended_shift}
The metric entropy of the $(m_-, m_+)$-extended shift $\shift_\phi$ (with respect to the measure $\med$ induced by the probability distribution $p^+$) is
    \begin{equation*}
    \enm(\shift_\phi) = \Enm(\mathcal C_0) = \sum_{s=0}^{m_+-1} -p^+_s \log(p^+_s).
    \end{equation*}
\end{introtheorem}

The probability distribution $p^+$ also induces a probability distribution $p^- = (p^-_0, \ldots, p^-_{m_- -1})$ on the backward alphabet $\{0, \ldots, m_- -1\}$, and these probability distributions are used to define ``quotient distributions'' $q^{s_-}$ for each $s_- \in \{0, \ldots, m_- -1\}$ (these are related to disintegrations, cf. \cref{ssec:disintegration} for details).

Our second main result relates the folding entropy $\mathcal{F}(\shift_\phi)$ of the extended shift with the metric entropy of the partitions $\mathcal C_0$ and $\mathcal C_{-1}$ by cylinders $\Cyl{0}{s}$ and $\Cyl{-1}{z}$, respectively, and explicitly to a formula involving the quotient distributions $q^{s_-}$.

\begin{introtheorem}
\label{thm:folding_entropy_formula}
The folding entropy of the $(m_-, m_+)$-extended shift $\shift_\phi$ (with respect to the measure $\med$ induced by the probability distribution $p^+$) is
	\begin{equation*}
    \mathcal{F}(\shift_\phi) = \Enm(\mathcal C_0) - \Enm(\mathcal C_{-1}) = \sum_{s_- = 0}^{m_--1} \Big(\smashoperator[r]{\sum_{s_+ \in \phi\inv(s_-)}} -q^{s_-}_{s_+} \log q^{s_-}_{s_+} \Big) p^-_{s_-}.
	\end{equation*}
\end{introtheorem}

The extended shift is in fact the standard bilateral Bernoulli shift when $m_+ = m_-$, which is invertible and has zero folding entropy. That illustrates how the folding entropy measures the non-invertibility of the system.

\notocsubsection{Structure of the paper}

\Cref{sec:preliminaries} contains basic nomenclature, notation and definitions from measure theory and ergodic theory that we will use throughout the paper.
In \Cref{sec:extended_shifts} we present the formal definition of the extended shift $\fun{\shift_\phi}{\Shift_{m_-,m_+}}{\Shift_{m_-,m_+}}$ and define their measure structure. In \cref{sec:proof_zip_shift_entropy} we calculate the Kolmogorov--Sinai entropy of the extended shifts. We find the general form for cylinder sets pulled-back by the extended shift dynamics, which involves proving a technically difficult \cref{prop:correfinamento.cilindros}, then use the Kolmogorov--Sinai theorem to prove \cref{thm:measure_entropy_of_extended_shift}. Finally, in \cref{sec:proof_folding_entropy_formula} we calculate the folding entropy of the extended shifts. For this we calculate the disintegration of the measure $\med$ with respect to the dynamical pullback of the atomic partition. After the technical work of \cref{ssec:quotient_measure,ssec:disintegration}, we obtain \cref{thm:folding_entropy_formula}.

\section{Preliminaries}
\label{sec:preliminaries}

We will denote the natural numbers (including $0$) by $\N$, the integers by $\Z$ and the real numbers by $\R$. We denote the strictly positive, positive, strictly negative, and negative integers by $\Z_{> 0}$, $\Z_{\geq 0}$, $\Z_{< 0}$, and $\Z_{\leq 0}$, respectively (and likewise for the other number sets).

\subsection{Measure spaces}

Let $X$ be a set. A \emph{$\sigma$-algebra} over $X$ is a family $\mens$ of subsets of $X$, whose elements are called \emph{measurable sets}, that contains the empty set and is closed under set complements and countable unions. The pair $(X, \mens)$ is called a \emph{measurable space}. Given any family $\mathcal{S}$ of subsets of $X$, the $\sigma$-algebra \emph{generated} by $\mathcal{S}$ is the smallest (relative to $\subseteq$) $\sigma$-algebra over $X$ that contains $\mathcal{S}$.

A \emph{measure} on $(X, \mens)$ is a function $\fun{\med}{\mens}{\R_{\geq 0}}$ that assigns the value $0$ to the empty set and is \emph{countably additive}, meaning that, for every pairwise disjoint countable family of measurable sets $(M_i)_{i \in \N}$,
	\begin{equation*}
	\med \big( \bigcup_{i \in \N} M_i \big) = \sum_{i \in \N} \med(M_i).
	\end{equation*}
The triple $(X, \mathcal M, \med)$ is called a \emph{measure space}. A \emph{probability measure} is a measure such that $\med(X)=1$, and the respective measure space is called a \emph{probability space}. We say that a property is valid for \emph{almost every point of $X$} when it is valid for every point of a subset of $X$ whose complement has measure $0$.

A \emph{measurable transformation} from a measure space $(X, \mens)$ to another $(X', \mens')$ is a transformation $\fun{f}{X}{X'}$ such that, for every measurable set $M' \in \mens'$, its inverse image by $f$ is measurable: $f\inv(M') \in \mens$. A \emph{measure-preserving transformation} from a measure space $(X, \mens, \med)$ to another $(X', \mens', \med')$ is a measurable transformation $\fun{f}{X}{X'}$ such that, for every measurable set $M' \in \mens'$, $\med(f\inv(M')) = \med'(M')$.

On a measure space, the \emph{integral} can be defined for functions $\fun{f}{X}{\R}$. We will denote the integral of $f$ with respect to $\med$ over a measurable set $M \subseteq X$ by $\int_M f \med$, or by $\int_{x \in M}f(x)\med(\dd x)$ when it is necessary to make the argument $x$ of $f$ explicit.

\subsection{Metric entropy}

Metric entropy was first defined by Kolmo\-go\-rov and Sinai and used as an invariant for dynamical systems over measure spaces \cite{Kolmogorov1958,Sinai1959}. Here we briefly define it and state the main theorem we will use in this work, the Kolmogorov--Sinai theorem (\cref{prop:Kolmogorov-SinaiTheorem}). We refer the reader to \cite{liv:VianaOliveira-FoundationsErgodicTheory} for the following definitions and any further information on metric entropy.

Let $(X, \mathcal M, \med)$ be a probability space. We will refer to any finite or countable family $\parti$ of pairwise disjoint measurable sets whose union has measure $1$ by a \emph{partition} of $X$. (This is similar to the usual definition of a partition, but weakened by the measure structure of the space). This defines, for almost every point $x \in X$, a unique set $\proj_{\parti}(x) \in \parti$ such that $x \in \proj_{\parti}(x)$, 
and hence a (almost everywhere defined) projection $\fun{\proj_{\parti}}{X}{\parti}$.

A partition $\parti$ is \emph{coarser} than a partition $\parti'$ (or $\parti'$ is \emph{finer} than $\parti$) when, for every element $P' \in \parti'$, there exists an element $P \in \parti$ such that $\med(P' \setminus P) = 0$ (which means that almost every point of $P'$ is contained in $P$). This is denoted by $\parti \refin \parti'$. We can also define an operation on the partitions: to each (finite or countable) family of partitions $(\parti_n)_{n \in \N}$, its \emph{correfinement} is
	\begin{equation*}
	\bigcorref_{i \in \N} \parti_n := \set{\bigcap_{n \in \N} P_n}{P_n \in \parti_n \text{ for each $n \in \N$}}.
	\end{equation*}
When we have only $2$ (of finitely many) partitions, we denote their correfinement by $\parti \corref \parti'$. The correfinement of a family of partitions is the smallest partition, relative to $\refin$, that is larger than every partition of the family.

The \emph{entropy} of $\parti$ is defined as
	\begin{equation}
	\label{eq:entropy_of_countable_partition}
	\Enm(\parti) := \sum_{P \in \parti} -\med(P)\log(\med(P)).
	\end{equation}
(Here and in what follows, we always assume that $0 \log 0 = 0$.)

Now let $\fun{f}{M}{M}$ be a measure-preserving transformation on $(X, \mathcal M, \med)$. We can define the \emph{pullback} of a partition $\parti$ by $f$ as
	\begin{equation*}
	f\inv(\parti) := \set{f\inv(P)}{P \in \parti}.
	\end{equation*}
This is also a partition in our specific sense. Then, for each $n \in \N$, the \emph{$n$-th dynamical correfinement} of $\parti$ is
	\begin{equation}
	\label{def:dynamical_correfinement}
	\parti^n := \bigcorref_{i=0}^{n-1} f^{-i}(\parti)
	\end{equation}
and the \emph{$n$-th bilateral dynamical correfinement} of $\parti$ is
	\begin{equation*}
	\parti^{\pm n} := \bigcorref_{i=-n}^{n-1} f^{-i}(\parti).
	\end{equation*}
An element $Q \in \parti^n$ is of the form $Q = P_0 \cap f\inv(P_1) \cap \cdots \cap f^{-(n-1)}(P_{n-1})$, for $P_i \in \parti$, and a point $x \in X$ belongs to $Q$ if, and only if, for every $0 \leq i \leq n-1$, $f^i(x) \in P_i$. This shows that the elements of $\parti^n$ partition the space into points which have the same orbit under $f$ for $n$ units of time.

The \emph{entropy of $f$} relative to $\parti$ is the limit
	\begin{equation*}
	\enm(f, \parti) := \lim_{n \to \infty} \frac{1}{n} \Enm(\parti^n).
	\end{equation*}
(Notice that $\parti^n$ depends on $f$ even though the notation does not make it explicit). The \emph{entropy} of $f$ is then the supremum of the entropies relative to all partitions with finite entropy (or, equivalently, finite partitions):
	\begin{equation*}
	\enm(f) := \sup_{\parti} \enm(f, \parti).
	\end{equation*}
This definition is very abstract and requires information about every finite partition, but there is a way to calculate the entropy of a transformation using only a sequence of partitions that have a special property. This is the content of the following theorem,
which we are going to use to obtain \cref{thm:measure_entropy_of_extended_shift}. The proof can be found in \cite{liv:VianaOliveira-FoundationsErgodicTheory}.

\begin{theorem}[Kolmogorov-Sinai]
\label{prop:Kolmogorov-SinaiTheorem}
Let $(X, \mens, \med)$ be a probability space, $\fun{f}{X}{X}$ a measure-preserving transformation and $(\parti_n)_{n \in \N}$ be an increasing sequence of partitions%
	\footnote{
		That is, for every $n, m \in \N$, if $n \leq m$ then $\parti_n \refin \parti_m$.
	}
with finite entropy such that $\bigcup_{i \in \N} \parti_i$ generates $\mens$ (up to measure $0$).
Then
	\begin{equation*}
	\enm(f) = \lim_{n \to \infty} \enm(f, \parti_n).
	\end{equation*}
\end{theorem}

\subsection{Disintegration of measure}

Given a probability space $(X, \mens, \med)$ and a partition $\parti$ (we do not require the partition to be countable here), we have the (almost everywhere defined) natural projection $\fun{\proj_{\parti}}{X}{\parti}$. Using $\proj_{\parti}$ we can pushforward a probability space structure onto $\parti$, namely $(\parti, \hat\mens, \hat\med)$, in which
	\begin{equation*}
	\hat \mens := \set{\mathcal{Q} \subseteq \parti}{\proj_{\parti}\inv(\mathcal{Q}) \in \mens}
	\end{equation*}
is the pushforward $\sigma$-algebra (or quotient $\sigma$-algebra) and
	\begin{equation*}
	\hat \med (\mathcal{Q}) := \med(\proj_{\parti}\inv(\mathcal{Q})) \quad\text{for every $\mathcal{Q} \in \hat \mens$}
	\end{equation*}
is the pushforward measure (or quotient measure).

\begin{definition}
Let $(X, \mens, \med)$ be a probability space and $\parti$ a partition of $X$. A \emph{disintegration} of $\med$ with respect to $\parti$ is a family of probability measures $(\med_P)_{P \in \parti}$ on $X$ such that
	\begin{enumerate}
	\item For almost every $P \in \parti$, $\med_P(P) = 1$;
	\item For every measurable set $M \in \mens$, the transformation $\parti \to \R$, $P \mapsto \med_P(M)$ is measurable;
	\item For every measurable set $M \in \mens$,
		\begin{equation}
		\label{eq:disintegration_formula}
		\med(M) = \int_{P \in \parti} \med_P(M)\hat\med(\dd P).
		\end{equation}
	\end{enumerate}
\end{definition}

Intuitively, this describes the way we can relate the Lebesgue measure on a square with the Lebesgue measure on each of its vertical sections by integration using Fubini's theorem.

Under certain conditions on the partition $\parti$, the disintegration of a measure is unique up to measure zero and always exists \cite{liv:VianaOliveira-FoundationsErgodicTheory}.

\subsection{Conditional and folding entropies}
\label{sec:conditional_entropy_and_folding_entropy}

Besides defining the entropy of a partition as in \cref{eq:entropy_of_countable_partition}, we can also define the conditional entropy of a partition $\parti$ relative to a partition $\parti'$. We follow the approach of \cite{art:Rokhlin-LecturesEntropyTheoryMeasurePreservingTransformations}. First we define, for each $P' \in \parti'$, the \emph{partition induced by $\parti$ on $P'$} as
	\begin{equation*}
	\parti|_{P'} := \set{P \cap P'}{P \in \parti}.
	\end{equation*}
Then the \emph{conditional entropy of $\parti$ with respect to $\parti'$} is defined \cite[Section~5.1]{art:Rokhlin-LecturesEntropyTheoryMeasurePreservingTransformations} using the disintegration of the measure $\med$ with respect to $\parti'$ by 
	\begin{equation}
	\label{eq:conditional_entropy_using_disintegration}
	\Enm(\parti \mid \parti') = \int_{P' \in \parti'} \Enm[\med_{P'}](\parti|_{P'}) \med_{\parti'}(\dd P').
	\end{equation}

This is a more general definition that works for non-countable partitions. In the case that the partitions are countable, we obtain the simplified formula presented in \cite{liv:VianaOliveira-FoundationsErgodicTheory}.

In \cite{Ruelle} the author introduces the \textit{folding entropy} for $\Cont^1$ transformations. It can be defined \cite{art:Liu-RuelleInequalityRelatingEntropyFoldingEntropyNegativeLyapunovExponents,Wu-Zhu} as the conditional entropy of the \emph{atomic partition}
	\begin{equation*}
	\epsilon := \set{\{x\}}{x \in X}
	\end{equation*}
with respect to its \emph{dynamical pullback}
	\begin{equation*}
	f\inv(\epsilon) = \set{f\inv(x)}{x \in X}.
	\end{equation*}

\begin{definition}
\label{def:folding_entropy}
Let $X$ be a probability space and $\fun{f}{X}{X}$ a measure-preserving transformation. The \emph{folding entropy} of $f$ with respect to $\med$ is
	\begin{equation*}
	\mathcal F(f) := \Enm(\epsilon \mid f\inv(\epsilon)).
	\end{equation*}
\end{definition}

\section{Extended shifts}
\label{sec:extended_shifts}

The extended shifts, also known as zip shifts, are a generalization of bilateral symbolic shifts. They were first introduced in \cite{Mehdipour-Lamei2, Mehdipour-Lamei1} and  expanded on in \cite{Mehdipour-Martins}. Instead of a single set of symbols ranging in $\{0, \ldots, m -1\}$, used to compose a bilateral symbolic sequence $x = (x_i)_{i \in \Z} \in \{0, \ldots, m -1\}^{\Z}$, we consider bilateral sequences that have symbols ranging in $\{0, \ldots, m_+ -1\}$ on their positive part and in $\{0, \ldots, m_- -1\}$ on their negative part.
To be able to still define the shift transformation, a function that translates one type of symbols to the other is needed. \Cref{def:extended_shift_space} formalizes this construction.

Let us first recall that, for $m \in \Z_{>0}$, we define the \emph{symbolic spaces} as $\Shift_{m} := \{0, \ldots, m-1\}^{\Z}$, $\Shift^+_{m} := \{0, \ldots, m-1\}^{\Z_{\geq 0}}$ and $\Shift^-_{m} := \{0, \ldots, m-1\}^{\Z_{<0}}$, the \emph{bilateral shift map} as $\fun{\shift}{\Shift_m}{\Shift_m}, (\ldots, x_{-1}; x_0, x_1, \ldots) \mapsto (\ldots, x_{-1}, x_0; x_1, \ldots)$, the \emph{unilateral shift} as $\fun{\shift}{\Shift^+_m}{\Shift^+_m}, (x_0, x_1, \ldots) \mapsto (x_1, x_2, \ldots)$, and, for each $s \in \{0, \ldots, m -1\}$, the \emph{inverted unilateral shift map with inserted symbol} $s$ as $\fun{\shift\inv_s}{\Shift^+_m}{\Shift^+_m}, (x_0, x_1, \ldots) \mapsto (s, x_0, x_1, \ldots)$. The unilateral shift and inverted unilateral shift are defined analogously on $\Shift^-_{m}$.

\begin{definition}
\label{def:extended_shift_space}
Let $m_-, m_+ \in \Z_{>0}$ be strictly positive integers and $\fun{\phi}{ \{0, \ldots, m_+ -1\} }{ \{0, \ldots, m_--1\} }$ a surjective function. The \emph{$(m_-, m_+)$-extended shift} dynamical system (with \emph{transition function} $\phi$) is the system $(\Shift_{m_-,m_+}, \shift_\phi)$ consisting of
    \begin{enumerate}
        \item the \emph{$(m_-, m_+)$-symbolic space} $\Shift_{m_-,m_+} := \Shift^-_{m_-} \times \Shift^+_{m_+}$, whose elements are pairs $x = (x^-, x^+)$ with $x^- = (\ldots, x_{-2}, x_{-1})$ and $x^+ = (x_0, x_1, \ldots)$, which can be identified with bilateral sequences
            \begin{equation*}
                x = (\ldots, x_{-2}, x_{-1}; x_0, x_1, \ldots)
            \end{equation*}
        with $x_{-1}, x_{-2}, \ldots \in \{0, \ldots, m_- -1\}$ and $x_0, x_1, \ldots \in \{0, \ldots, m_+ -1\}$;

        \item the \emph{$(m_-, m_+)$-extended shift} $\fun{\shift_\phi}{\Shift_{m_-,m_+}}{\Shift_{m_-,m_+}}$, which acts on each $x \in \Shift_{m_-,m_+}$ by
            \begin{equation*}
                \shift_{\phi}(x) := (\shift\inv_{\phi(x_0)}(x^-), \shift(x^+)) = (\ldots, x_{-1}, \phi(x_0); x_1, x_2, \ldots)
            \end{equation*}
    \end{enumerate}
\end{definition}

We can also define a notion of an extended Bernoulli transformation.

\begin{definition}
A measure-preserving map $\fun{f}{X}{X}$ defined on a Lebesgue space is an $(m_-, m_+)$-Bernoulli transformation if it is isomorphic (mod 0) to an $(m_-, m_+)$-extended shift $\shift_\phi$.
\end{definition} 

The $2$-to-$1$ baker's transformation defined in \cite{Mehdipour-Martins} exemplifies a $(2,4)$-Bernoulli transformation. We omit the formal definition here, but \cref{fig:baker_2-to-1} shows how this transformation is defined on the square $Q$ in three steps, and \cref{fig:baker_2-to-1_partition+partitions_iterated} shows the partitions of the square that are used to encode the system and obtain the isomorphism to a $(2,4)$-extended shift and how these partitions iterate under the action of the dynamics over time.

\begin{figure}
	\centering
	\includesvg{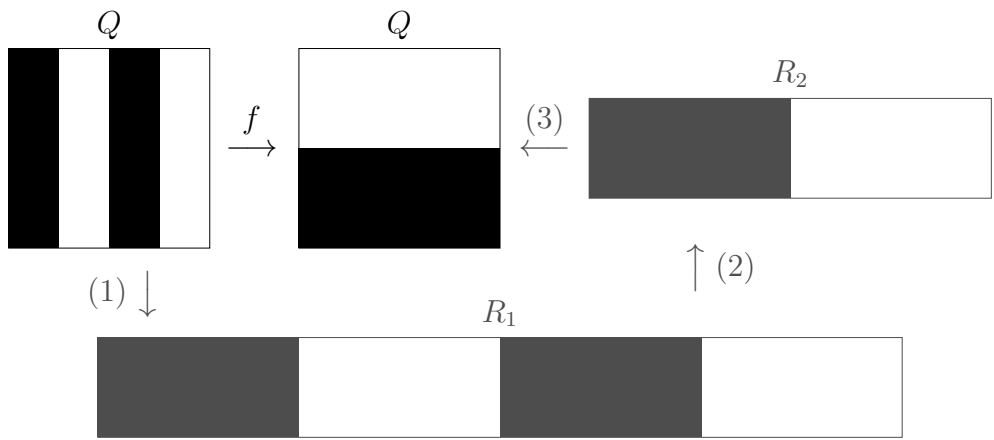}
	\caption[The $2$-to-$1$ baker's transformation.]{
		The $2$-to-$1$ baker's transformation.
		On step (1), the square $Q$ is dilated by $4$ in the horizontal direction and by $\frac{1}{2}$ in the vertical direction, resulting in a rectangle $R_1$.
		On step (2), the right-most bands of the rectangle $R_1$ are glued over the left-most ones, resulting in a smaller rectangle $R_2$.
		On step (3), the left right band of the rectangle $R_2$ is glued above the left band, resulting again in the square $Q$.
	}
	\label{fig:baker_2-to-1}
\end{figure}

\begin{figure}
	\centering
	\includesvg{images/baker_2-to-1_partition+partitions_iterated}
	\caption[The horizontal and vertical partitions of the $2$-to-$1$ baker's transformation and their iterations.]{
		The first $2$ iterations of the horizontal partition $\mathcal H = \{H_a, H_b\}$ and the vertical partition $\mathcal V = \{V_0, V_1, V_2, V_3\}$ of the $2$-to-$1$ baker's transformation. In the limit, the iterations of the horizontal partition is the partition of the square by horizontal line, and the iteration of the vertical partition is the partition by vertical lines.
	}
	\label{fig:baker_2-to-1_partition+partitions_iterated}
\end{figure}

\subsection{Measurable structure}
\label{ssec:measurable_structure}

The $\sigma$-algebra $\mathcal B$ of the space $\Shift_{m_-,m_+}$ is the one generated by cylinder sets: for each $i \in \Z$, and for each $s \in \{0, \ldots, m_--1\}$ if $i < 0$ or $s \in \{0, \ldots, m_+ -1\}$ if $i \geq 0$, we define the \emph{cylinder}
    \begin{equation*}
    \Cyl{i}{s} := \set{x \in \Shift_{m_-,m_+}}{x_i = s}.
    \end{equation*}
and denote
    \begin{equation*}
    \Cyl{i_1,\ldots,i_k}{s_1,\ldots, s_k} := \set{x \in \Shift_{m_-,m_+}}{x_{i_1} = s_1, \ldots, x_{i_k} = s_k} = \Cyl{i_1}{s_1} \cap \cdots \cap \Cyl{i_k}{s_k}.
    \end{equation*}
For simplicity, when the $i_j$'s increase by $1$, we just denote this by $\Cyl{i_1}{s_1, \ldots, s_k}$.
We also define, for $s_+ \in \{0, \ldots, m_+-1\}$, the \emph{extended cylinder}
	\begin{equation*}
	\Cyl{i}{\phi\inv(s_+)} := \bigcup_{s_- \in \phi\inv(s_+)} \Cyl{i}{s_-}.
	\end{equation*}

The next proposition shows how the dynamics acts backwards and forwards on cylinders.

\begin{proposition}
\label{prop:extended_shift_iterated_cylinders}
Let $k \in \Z_{\geq 0}$ and $0 \leq s < \max\{m_-,m_+\}$. Then
    \begin{equation*}
   	\shift_\phi^{-k}(\Cyl{i}{s}) =
   		\begin{cases}
   		\Cyl{i+k}{s}, & i \notin [-k, -1] \cap \Z \\
   		\Cyl{i+k}{\phi\inv(s)}, & i \in [-k, -1] \cap \Z.
   		\end{cases}
   	\end{equation*}
and
    \begin{equation*}
   	\shift_\phi^k(\Cyl{i}{s}) =
   		\begin{cases}
   		\Cyl{i-k}{s}, & i \notin [0, k-1] \cap \Z \\
   		\Cyl{i-k}{\phi(s)}, & i \in [0, k-1] \cap \Z.
   		\end{cases}
   	\end{equation*}
\end{proposition}
\begin{proof}
For the inverse image, it holds that
	\begin{equation*}
	\shift_\phi\inv(\Cyl{i}{s}) =
		\begin{cases}
		\Cyl{i+1}{s}, & i \neq -1 \\
		\Cyl{0}{\phi\inv(s)}, & i = -1
		\end{cases}
	\end{equation*}
and, for the direct image, it holds that
	\begin{equation*}
	\shift_\phi(\Cyl{i}{s}) =
		\begin{cases}
		\Cyl{i-1}{s}, & i \neq 0 \\
		\Cyl{-1}{\phi(s)}, & i = 0.
		\end{cases}
	\end{equation*}
Then, by induction on $k$, we obtain the statement.
\end{proof}

\subsection{Measure structure}
\label{ssec:measure_structure}

In order to define a measure on $(\Shift_{m_-,m_+}, \mathcal B)$, it is sufficient to define it on the cylinders $\Cyl{i}{s}$. We start with a probability measure $\med^+$ on the symbol set $\{0, \ldots, m_+ -1\}$. Since $\{0, \ldots, m_+ -1\}$ is a finite set with atomic $\sigma$-algebra, this probability measure can be identified with a discrete probability distribution
$p^+ = (p^+_0, \ldots, p^+_{m_+ -1})$ (that is, for every $s \in \{0, \ldots, m_+ -1\}$ we have $p^+_{s} \in \R_{\geq 0}$, and $\sum_{s=0}^{m_+ -1} p^+_s = 1$) 
by defining, for each $0 \leq s < m_+$,
	\begin{equation*}
	p^+_{s} := \med^+(\{s\}).
	\end{equation*}

Using the surjective transition function $\fun{\phi}{\{0, \ldots, m_+ -1\}}{\{0, \ldots, m_- -1\}}$, we can pushforward this probability measure $\med^+$ to the probability measure $\med^- := \phi\emp \med^+$ on $\{0, \ldots, m_- -1\}$. This is done by considering the partition $\{\phi\inv(z)\}_{z \in \{0, \ldots, m_--1\}}$ of $\{0, \ldots, m_+ -1\}$ by the inverse images of elements of $\{0, \ldots, m_--1\}$. The pushforward measure of $\{z\} \subseteq \{0, \ldots, m_--1\}$ is then the sum of the measure of all the elements of $\phi\inv(z)$ on $\{0, \ldots, m_+ -1\}$, given for each $z \in \{0, \ldots, m_--1\}$ by
	\begin{equation*}
	\med^-(\{z\}) = \phi\emp \med^+(\{z\}) = \med^+(\phi\inv(\{z\})) = \sum_{s \in \phi\inv(z)} \med^+(\{s\}).
	\end{equation*}
In the same way as we did for $p^+$, we can identify the measure $\med^-$ with a probability distribution $p^- = (p^-_0, \ldots, p^-_{m_--1})$ by setting, for each $0 \leq z < m_-$,
	\begin{equation*}
	p^-_{z} := p^-(\{z\}).
	\end{equation*}
Then, for a cylinder $\Cyl{i}{s}$, we can define its measure as $p^+_s$ if $i \geq 0$ and $p^-_s$ if $i<0$.

\begin{definition}
\label{def:extended_shift_measure}
Let $(\Shift_{m_-,m_+}, \shift_\phi)$ be an $(m_-, m_+)$-extended shift space, $p^+$ a probability measure on $\{0, \ldots, m_+ -1\}$ and $p^- = \phi\emp p^+$ the pushforward probability measure on $\{0, \ldots, m_--1\}$.
The \emph{probability measure} on $(\Shift_{m_-,m_+}, \shift_\phi)$ induced by $p^+$ is the probability measure $\fun{\med}{\mathcal B}{\intff{0}{1}}$
defined on cylinders by
	\begin{equation*}
	\med(\Cyl{i}{s}) :=
		\begin{cases}
		p^-_s, & i < 0 \\
		p^+_s, & i \geq 0
		\end{cases}
	=
		\begin{cases}
		\displaystyle\sum_{s' \in \phi\inv(s)} p^+_{s'}, & i < 0 \\
		p^+_s, & i \geq 0.
		\end{cases}
	\end{equation*}
\end{definition}

From the way we defined the measure on $\{0, \ldots, m_--1\}$ by the pushforward, it is easy to show that the extended shift dynamics is measure-preserving. We just need to be careful considering the different cases.

\begin{proposition}
\label{prop:extended_shift_preserves_measure}
Let $(\Shift_{m_-,m_+}, \shift_\phi)$ be an $(m_-, m_+)$-extended shift space and $p^+$ a probability measure on $\{0, \ldots, m_+ -1\}$. The dynamics $\shift_\phi$ preserves the measure $\med$.
\end{proposition}
\begin{proof}
It suffices to show that, for every basic cylinder $\Cyl{i}{s}$,
	\begin{equation*}
	\med(\shift_\phi\inv(\Cyl{i}{s})) = \med(\Cyl{i}{s}).
	\end{equation*}
We consider three cases:
\begin{enumerate}
\item ($i \geq 0$) In this case, $\shift_\phi^{-1}(\Cyl{i}{s}) = \Cyl{i+1}{s}$ (\cref{prop:extended_shift_iterated_cylinders}). Since $i+1 \geq 1$, if follows from \cref{def:extended_shift_measure} that
	\begin{equation*}
	\med(\shift_\phi\inv(\Cyl{i}{s})) = \med(\Cyl{i+1}{s}) = p^+_s = \med(\Cyl{i}{s}).
	\end{equation*}

\item ($i < -1$) In this case, it also holds that $\shift_\phi\inv(\Cyl{i}{s}) = \Cyl{i+1}{s}$ (\cref{prop:extended_shift_iterated_cylinders}). Since $i+1 < 0$, i follows from \cref{def:extended_shift_measure} that
	\begin{equation*}
	\med(\shift_\phi\inv(\Cyl{i}{s})) = \med(\Cyl{i+1}{s}) = p^-_s = \med(\Cyl{i}{s}).
	\end{equation*}

\item ($i = -1$) In this case, $\shift_\phi\inv(\Cyl{-1}{s}) = \Cyl{0}{\phi\inv(s)} = \bigcup_{s' \in \phi\inv(s)} \Cyl{0}{s'}$ (\cref{prop:extended_shift_iterated_cylinders}). Since $i+1 = 0$, it follows from \cref{def:extended_shift_measure} that
	\begin{equation*}
	\med(\shift_\phi\inv(\Cyl{i}{s})) = \med(\ \ \bigcup_{\mathclap{s' \in \phi\inv(s)}} \Cyl{0}{s'} \ \ ) = \sum_{s' \in\phi\inv(s)} \med(\Cyl{0}{s'}) = \sum_{s' \in\phi\inv(s)} p^+_{s'} = \med(\Cyl{i}{s}).
	\qedhere
	\end{equation*}
\end{enumerate}
\end{proof}

\section{Metric entropy of extended shifts}
\label{sec:proof_zip_shift_entropy}

\subsection{Partitions by cylinders}

We begin by defining some basic partitions of our space.

\begin{definition}
	Let $i \in \Z$. The \emph{partition by cylinders of index $i$} is the partition
	\begin{equation*}
		\mathcal C_i :=
		\begin{cases}
			\set{\Cyl{i}{s}}{s \in \{0, \ldots, m_+-1\}} & i \geq 0 \\
			\set{\Cyl{i}{s}}{s \in \{0, \ldots, m_--1\}} & i < 0. \\
		\end{cases}
	\end{equation*}
	Let $n, n' \in \Z$. The \emph{partition by cylinders of indices from $n$ to $n'$} is the partition
	\begin{equation*}
		\mathcal C_{n, \ldots, n'} := \bigcorref_{i = n}^{n'} \mathcal C_i.
	\end{equation*}
\end{definition}
~
The following simple \namecref{prop:extended_shift_iterated_partitions} sums up how the dynamics of the shift $\shift_{\phi}$ acts on these partitions.

\begin{lemma}
	\label{prop:extended_shift_iterated_partitions}
	For every $i \geq 0$,
	\begin{enumerate}
		\item $\shift_{\phi}^i(\mathcal C_0) = \mathcal C_{-i}$;
		\item $\shift_{\phi}^{-i}(\mathcal C_0) = \mathcal C_{i}$;
		\item $\shift_{\phi}^{-i}(\mathcal C_{-(i+1)}) = \mathcal C_{-1}$;
		\item $\mathcal C_0^n = \mathcal C_{0,\ldots, n-1}$.
		\item $\mathcal C_0^{\pm n} = \mathcal C_{-n,\ldots, n-1}$.
	\end{enumerate}
\end{lemma}
\begin{proof}
	This is a consequence of \cref{prop:extended_shift_iterated_cylinders}.
	\begin{enumerate}
		\item Since $\shift_{\phi}(\Cyl{0}{s}) = \Cyl{-1}{\phi(s)}$ and $\phi$ is surjective, it follows that $\shift_{\phi}(\mathcal C_0) = \mathcal C_{-1}$. By induction, $\shift_{\phi}^i(\mathcal C_0) = \mathcal C_{-i}$.
		
		\item Since $\shift_{\phi}\inv(\Cyl{0}{s}) = \Cyl{1}{s}$, it follows that $\shift_{\phi}\inv(\mathcal C_0) = \mathcal C_1$. By induction, $\shift_{\phi}^{-i}(\mathcal C_0) = \mathcal C_{i}$.
		
		\item Since $\sigma^{-1}(\Cyl{-2}{s}) = \Cyl{-1} {s}$, it follows that $\shift_{\phi}^{-1}(\mathcal C_{-2}) = \mathcal C_{-1}$. By induction, $\shift_{\phi}^{-i}(\mathcal C_{-(i+1)}) = \mathcal C_{-1}$.
		
		\item It follows that
			\begin{equation*}
			\mathcal C_0^n = \bigcorref_{i=0}^{n-1} \shift_{\phi}^{-i}(\mathcal C_0) = \bigcorref_{i=0}^{n-1} \mathcal C_i.
			\end{equation*}
		\item It follows that
			\begin{equation*}
			\mathcal C_0^{\pm n} = \bigcorref_{i=-n}^{n-1} \shift_{\phi}^{-i}(\mathcal C_0) = \bigcorref_{i=-n}^{n-1} \mathcal C_i.
			\qedhere
			\end{equation*}
	\end{enumerate}
\end{proof}

\subsection{Calculating the metric entropy}
\label{ssec:entropy_of_extended_shift}

We now calculate the metric entropy of $(\Shift_{m_-, m_+}, \shift_{\phi})$ and relate it to the entropy of the probability distributions $p^+$ and $p^-$. We start with the partitions $\mathcal C_0$ and $\mathcal C_{-1}$.

\begin{lemma}
\label{prop:entropy_basic_cilinder_partitions}
$\Enm(\mathcal C_0) = \sum_{s=0}^{m_+-1} -p^+_{s} \log p^+_{s}$
and
$\Enm(\mathcal C_{-1}) = \sum_{s=0}^{m_--1} -p^-_s \log p^-_s$.
\end{lemma}
\begin{proof}
It follows directly from \cref{eq:entropy_of_countable_partition} from the simple calculations
	\begin{equation*}
	\Ens(\mathcal C_0) = \sum_{s=0}^{m_+-1} -\med(\Cyl{0}{s}) \log(\med(\Cyl{0}{s})) = \sum_{s=0}^{m_+-1} -p^+_{s} \log p^+_{s}.
	\end{equation*}
and
	\begin{equation*}
	\Ens(\mathcal C_{-1}) = \sum_{s=0}^{m_--1} -\med(\Cyl{-1}{s}) \log(\med(\Cyl{-1}{s})) = \sum_{s=0}^{m_--1} -p^-_s \log p^-_s.
	\qedhere
	\end{equation*}
\end{proof}

This shows, as could be expected, that the entropy of the partition $\mathcal C_0$ is related to $p^+$, the distribution of the positive part of the extended shift $\Shift_{m_-, m_+}$, while the entropy of the partition $\mathcal C_{-1}$ is related to $p^-$, the distribution of the negative part of $\Shift_{m_-,m_+}$. 

We can now calculate the metric entropy of a partition by cylinders other than the basic $\mathcal C_0$ and $\mathcal C_{-1}$.

\begin{lemma}
\label{prop:entropia.cilindros}
$\Enm(\mathcal C_{-n, \ldots,0, \ldots, n'-1}) = n \Enm(\mathcal C_{-1}) + n' \Enm(\mathcal C_0)$.
\end{lemma}
\begin{proof}
For every $i \geq 1$, it holds that $\mathcal C_i = \shift_{\phi}^{-i}(\mathcal C_0)$ and $\shift_{\phi}^{-i}(\mathcal C_{-(i+1)}) = \mathcal C_{-1}$ (\cref{prop:extended_shift_iterated_partitions}). Since $\shift_{\phi}$ preserves the measure $\med$ (\cref{prop:extended_shift_preserves_measure}), it follows that $\Enm(\mathcal C_i) = \Enm(\mathcal C_0)$ and $\Enm(\mathcal C_{-(i+1)}) = \Enm(\mathcal C_{-1})$.

Besides that, for any integers $i < i'$, the partitions $\mathcal C_i$ and $\mathcal C_{i'}$ are independent, because $\Cyl{i}{s} \cap \Cyl{i'}{s'} = \Cyl{i, i'}{s, s'}$ and $\med(\Cyl{i,i'}{s,s'}) = \med(\Cyl{i}{s})\med(\Cyl{i'}{s'})$. Thus it follows that
	\begin{equation*}
	\Enm(\mathcal{C}_{-n, \ldots,0, \ldots, n'-1}) = \Enm\left( \bigcorref_{i=-n}^{n'-1} \mathcal C_i \right) = \sum_{i=-n}^{n'-1} \Enm(\mathcal C_i) = n \Enm(\mathcal C_{-1}) + n' \Enm(\mathcal C_0).
	\qedhere
	\end{equation*}
\end{proof}

In particular, since $\mathcal C_0^{n} = \mathcal C_{0, \ldots, n-1}$ (\cref{prop:extended_shift_iterated_partitions}), this implies that
	\begin{equation*}
	\enm(\shift_{\phi}, \mathcal C_0) = \lim_{n \to \infty} \frac{1}{n} \Enm(\mathcal C_0^{n}) = \lim_{n \to \infty} \frac{1}{n} n \Enm(\mathcal C_0) = \Enm(\mathcal C_0).
	\end{equation*}

To calculate the metric entropy of the system, we will use the Kolmogorov-Sinai theorem (\cref{prop:Kolmogorov-SinaiTheorem}). To that end we define a sequence of partitions.

\begin{definition}
$\parti_n := \mathcal C_0^{\pm n} = \mathcal C_{-n, \ldots, n-1}$.
\end{definition}

We will eventually need to use the metric entropy of $\parti_n^k$ (check \cref{def:dynamical_correfinement}), the $k$-th dynamical correfinement of the partition $\parti_n$, so the following \namecref{prop:correfinamento.cilindros} shows that it is just a partition by cylinders. The proof is trickier than would be expected.

\begin{lemma}
\label{prop:correfinamento.cilindros}
Let $n \geq 1$ and $k \geq 2n$. Then $\parti_n^k = \mathcal C_{-n, \ldots, n+k-2}$.
\end{lemma}
\begin{proof}
The dynamical correfinement of $\parti_n$ is defined by $\parti_n^k = \bigcorref_{j=0}^{k-1} \shift_{\phi}^{-j}(\parti_n)$, so let us first calculate a generic element of the pullback partition
	\begin{align*}
	\shift_{\phi}^{-j}(\parti_n) = \set{\shift_{\phi}^{-j}(C)}{C \in \parti_n}.
	\end{align*}
Each cylinder of $\parti_n = \mathcal C_{-n, \ldots, n-1}$ has the form
	\begin{equation*}
	\Cyl{-n, \ldots, n-1}{s_{-n}, \ldots, s_{n-1}} = \bigcap_{i=-n}^{n-1} \Cyl{i}{s_i},
	\end{equation*}
with $s_i \in \{0, \ldots, m_--1\}$ if $i < 0$ and $s_i \in \{0, \ldots, m_+-1\}$ if $i \geq 0$. Then
	\begin{equation*}
	\shift_{\phi}^{-j}(\Cyl{-n, \ldots, n-1}{s_{-n}, \ldots, s_{n-1}}) = \shift_{\phi}^{-j} (\bigcap_{i=-n}^{n-1} \Cyl{i}{s_i}) = \bigcap_{i=-n}^{n-1} \shift_{\phi}^{-j} (\Cyl{i}{s_i}).
	\end{equation*}
Based on \cref{prop:extended_shift_iterated_cylinders}, we can separate this in $3$ intersections%
	\footnote{
		In order to simplify notation, we define that intersections that have the top index strictly smaller than the bottom index should be consider to be the whole space $\Shift_{m_-,m_+}$, so that they can be ignored. In \cref{eq:intersection_of_pulledback_cylinders_grouped}, this happens for the first intersection in the case $j > n-1$ (or equivalently $-(j+1) < -n$) and for the second intersection in the case $j=0$ (or equivalently $-1 < -j$).
	}
as follows:
	\begin{equation}
	\label{eq:intersection_of_pulledback_cylinders_grouped}
	\begin{split}
	\shift_{\phi}^{-j}(&\Cyl{-n, \ldots, n-1}{s_{-n}, \ldots, s_{n-1}}) = \\
        &= \bigcap_{i=-n}^{-(j+1)} \shift_{\phi}^{-j} (\Cyl{i}{s_i}) \cap \bigcap_{i=-j}^{-1} \shift_{\phi}^{-j}(\Cyl{-1}{s_{-1}}) \cap \bigcap_{i=0}^{n-1} \shift_{\phi}^{-j} (\Cyl{i}{s_i})  \\
		&= \bigcap_{i=-n}^{-(j+1)} \Cyl{i+j}{s_i} \cap \bigcap_{i=-j}^{-1} \Cyl{i+j}{\phi\inv(s_i)} \cap \bigcap_{i=0}^{n-1} \Cyl{i+j}{s_i}.
	\end{split}
	\end{equation}

Notice that in \cref{eq:intersection_of_pulledback_cylinders_grouped}, for $-n \leq i \leq -(j+1)$ and $0 \leq i \leq n-1$ we have basic cylinders of the form $\Cyl{i+j}{s_i}$ and, for $-j \leq i \leq -1$, we have extended cylinders (unions of cylinders) of the form
	\begin{equation*}
	\Cyl{i+j}{\phi\inv(s_i)} = \bigcup_{s \in \phi\inv(s_i)} \Cyl{i+j}{s}.
	\end{equation*}
This shows that $\shift_{\phi}^{-j}(\parti_n)$ is not a partition by cylinders (unless $\phi$ is bijective and hence the sets $\phi\inv(s^j_i)$ are singletons, but this is just a regular shift, not the usual case for extended shifts).

We must now calculate a generic element of $\parti_n^k = \bigcorref_{j=0}^{k-1} \shift_{\phi}^{-j}(\parti_n)$.
To that end, for each $0 \leq j \leq k-1$ we take cylinders $C^j \in \parti_n$, defined by
	\begin{equation*}
	C^j := \Cyl{-n, \ldots, n-1}{s^j_{-n}, \ldots, s^j_{n-1}} = \bigcap_{i=-n}^{n-1} \Cyl{i}{s^j_i}
	\end{equation*}
with $s^j_i \in \{0, \ldots, m_--1\}$ if $i < 0$ and $s^j_i \in \{0, \ldots, m_+-1\}$ if $i \geq 0$. An element of $\parti_n^k$ is a non-empty set of the form $\bigcap_{j=0}^{k-1} \shift_{\phi}^{-j}(C^j)$.
From \cref{eq:intersection_of_pulledback_cylinders_grouped}, it follows that this set is given by
	\begin{equation}
	\label{eq:lemma_generic_element_of_Pnk}
    \begin{split}
	\bigcap_{j=0}^{k-1} &\shift_{\phi}^{-j}(\Cyl{-n, \ldots, n-1}{s^j_{-n}, \ldots, s^j_{n-1}}) = \\
        &= \bigcap_{j=0}^{k-1} \bigcap_{i=-n}^{-(j+1)} \Cyl{i+j}{s^j_i} \cap \bigcap_{j=0}^{k-1} \bigcap_{i=-j}^{-1} \Cyl{i+j}{\phi\inv(s^j_i)} \cap \bigcap_{j=0}^{k-1} \bigcap_{i=0}^{n-1} \Cyl{i+j}{s^j_i}.
    \end{split}
	\end{equation}

This shows that a generic element of $\parti_n^k$ (as in \cref{eq:lemma_generic_element_of_Pnk}) is an intersection of basic cylinders and extended cylinders (which are unions of basic cylinders). These cylinders on the right-hand side of \cref{eq:lemma_generic_element_of_Pnk} are indexed by $l := i+j$, which varies between $-n$ and $n-k-2$ since $j$ varies between $0$ and $k-1$, and $i$ varies between $-n$ and $n-1$.

We wish to find conditions on the symbols $s^j_i$ that guarantee the intersections in \cref{eq:lemma_generic_element_of_Pnk} is non-empty. For that, we will reorganize the intersections based on the indices $l$ and $j$. Define $B_l$ to be the intersection of every cylinder and extended cylinder in \cref{eq:lemma_generic_element_of_Pnk} that has index $l$. Thus
	\begin{equation}
	\label{eq:lemma_generic_element_of_Pnk_by_index_l}
	\bigcap_{j=0}^{k-1}  \shift_{\phi}^{-j}(\Cyl{-n, \ldots, n-1}{s^j_{-n}, \ldots, s^j_{n-1}}) = \bigcap_{l=-n}^{n+k-2} B_l,
	\end{equation}
and each set $B_l$ is an intersection that depends on a range of values of $j$. 

Since the intersection of a cylinder or extended cylinder with another cylinder or extended cylinder is non-empty if they have different indices, the intersection on the right-hand side of \cref{eq:lemma_generic_element_of_Pnk_by_index_l} is non-empty if, and only if, each $B_l \neq \emptyset$. In what follows we shall determine the range of $j$ for each $l$ and find conditions on the symbols $s^j_i$. We separate our analysis in many cases.

\begin{enumerate}
\item ($-n \leq l \leq -1$) In this case $0 \leq j \leq l+n$ and no extended cylinder occurs. In order to have $B_l \neq \emptyset$, all the relations in \cref{tab:cylinder_symbols_relations} must be satisfied, and hence
		\begin{equation}
		\label{eq:lemma_Bl_case1}
		B_l = \bigcap_{j = 0}^{l+n} \Cyl{l}{s^{j}_{l-j}} = \Cyl{l}{s^{0}_{l}}.
		\end{equation}

\item ($0 \leq l \leq n-1$) In this case, when $0 \leq j \leq l$ we have basic cylinders and when $l+1 \leq j \leq l+n$ we have extended cylinders. In order to have $B_l \neq \emptyset$, all the relations in \cref{tab:cylinder_symbols_relations} must be satisfied, and hence
		\begin{equation}
		\label{eq:lemma_Bl_case2}
		B_l = \bigcap_{j = 0}^{l} \Cyl{l}{s^{j}_{l-j}} \cap \bigcap_{j = l+1}^{l+n} \Cyl{l}{\phi\inv(s^{j}_{l-j})} = \Cyl{l}{s^{l}_{0}}.
		\end{equation}

\item ($n \leq l \leq k-n-1$) In this case, when $l-n+1 \leq j \leq l$ we have basic cylinders and when $l+1 \leq j \leq l+n$ we have extended cylinders. In order to have $B_l \neq \emptyset$, all the relations in \cref{tab:cylinder_symbols_relations} must be satisfied, and hence
		\begin{equation}
		\label{eq:lemma_Bl_case3}
		B_l = \bigcap_{j = l-n+1}^{l} \Cyl{l}{s^{j}_{l-j}} \cap \bigcap_{j = l+1}^{l+n} \Cyl{l}{\phi\inv(s^{j}_{l-j})} = \Cyl{l}{s^{l}_{0}}.
		\end{equation}

\item ($k-n \leq l \leq k-2$) In this case, when $l-n+1 \leq j \leq l$ we have basic cylinders and when $l+1 \leq j \leq k-1$ we have extended cylinders. In order to have $B_l \neq \emptyset$, all the relations in \cref{tab:cylinder_symbols_relations} must be satisfied, and hence
		\begin{equation}
		\label{eq:lemma_Bl_case4}
		B_l = \bigcap_{j = l-n+1}^{l} \Cyl{l}{s^{j}_{l-j}} \cap \bigcap_{j = l+1}^{k-1} \Cyl{l}{\phi\inv(s^{j}_{l-j})} = \Cyl{l}{s^{l}_{0}}.
		\end{equation}

\item ($k-1 \leq l \leq k+n-2$) In this case $l-n+1 \leq j \leq k-1$ and no extended cylinder occurs. In order to have $B_l \neq \emptyset$, all the relations in \cref{tab:cylinder_symbols_relations} must be satisfied, and hence
		\begin{equation}
		\label{eq:lemma_Bl_case5}
		B_l = \bigcap_{j = l-n+1}^{k-1} \Cyl{l}{s^{j}_{l-j}} = \Cyl{l}{s^{k-1}_{l-k+1}}.
		\end{equation}
\end{enumerate}

\begin{table}
	\centering
	\begin{tabular}{c r c l}
	\toprule

	$l$			& \multicolumn{3}{c}{\textbf{Relations}} \\
	\midrule

	$-n$		& \textcolor{imecc-red}{$s^{0}_{-n}$}	&	& \\
	$-n+1$		& \textcolor{imecc-red}{$s^{0}_{-n+1}$} $= s^{1}_{-n}$	&	& \\
	$\vdots$	& $\vdots \quad$	&	& \\
	$-1$		& \textcolor{imecc-red}{$s^{0}_{-n}$} $= \cdots = s^{n-1}_{-n}$	&	& \\
	
	\cmidrule(lr){2-4}

	$0$			& \textcolor{imecc-red}{$s^{0}_{0}$} & $\in$ & $\phi\inv(s^{1}_{-1}) = \cdots = \phi\inv(s^{n}_{-n})$ \\
	$1$			& $s^{0}_{1} = $ \textcolor{imecc-red}{$s^{1}_{0}$} & $\in$ & $\phi\inv(s^{2}_{-1}) = \cdots = \phi\inv(s^{n+1}_{-n})$ \\
	$\vdots$	&	& $\vdots$ & \\
	$n-1$		& $s^{0}_{n-1} = \cdots =$ \textcolor{imecc-red}{$s^{n-1}_{0}$} & $\in$ & $\phi\inv(s^{n}_{-1}) = \cdots = \phi\inv(s^{2n-1}_{-n})$ \\
	$n$			& $s^{1}_{n-1} = \cdots =$ \textcolor{imecc-red}{$s^{n}_{0}$} & $\in$ & $\phi\inv(s^{n+1}_{-1}) = \cdots = \phi\inv(s^{2n}_{-n})$ \\
	$\vdots$	&	& $\vdots$ & \\
	$k-1-n$		& $s^{k-2n}_{n-1} = \cdots =$ \textcolor{imecc-red}{$s^{k-1-n}_{0}$} & $\in$ & $\phi\inv(s^{k-n}_{-1}) = \cdots = \phi\inv(s^{k-1}_{-n})$ \\
	$k-n$		& $s^{k-2n+1}_{n-1} = \cdots =$ \textcolor{imecc-red}{$s^{k-n}_{0}$} & $\in$ & $\phi\inv(s^{k-n+1}_{-1}) = \cdots = \phi\inv(s^{k-1}_{-(n-1)})$ \\
	$\vdots$	&	& $\vdots$ & \\
	$k-2$		& $s^{k-1-n}_{n-1} = \cdots =$ \textcolor{imecc-red}{$s^{k-2}_{0}$}	& $\in$	& $\phi\inv(s^{k-1}_{-1})$ \\

	\cmidrule(lr){2-4}
	
	$k-1$		& $s^{k-n}_{n-1} = \cdots =$ \textcolor{imecc-red}{$s^{k-1}_{0}$}	&	& \\
	$k$			& $s^{k-n+1}_{n-1} = \cdots =$ \textcolor{imecc-red}{$s^{k-1}_{1}$}	&	& \\
	$\vdots$	& $\vdots \quad$	&	& \\
	$n+k-2$		& \textcolor{imecc-red}{$s^{k-1}_{n-1}$}	&	& \\

	\bottomrule
\end{tabular}
	\caption[Relations between the symbols $s^{j}_{i} = s^{j}_{l-j}$ from \cref{eq:lemma_generic_element_of_Pnk}]{
		Relations between the symbols $s^{j}_{i} = s^{j}_{l-j}$ from \cref{eq:lemma_generic_element_of_Pnk} for $-n \leq l \leq n+k-2$. For each $l$, the symbol in \textcolor{imecc-red}{red} determines every other symbol in that line.
	}
	\label{tab:cylinder_symbols_relations}
\end{table}

Thus using \cref{eq:lemma_Bl_case1,eq:lemma_Bl_case2,eq:lemma_Bl_case3,eq:lemma_Bl_case4,eq:lemma_Bl_case5} on \cref{eq:lemma_generic_element_of_Pnk_by_index_l}, if follows that
	\begin{equation*}
	\bigcap_{j=0}^{k-1} \shift_{\phi}^{-j}(\Cyl{-n, \ldots, n-1}{s^j_{-n}, \ldots, s^j_{n-1}}) = \bigcap_{l=-n}^{-1} \Cyl{l}{s^0_l} \cap \bigcap_{l=0}^{k-2} \Cyl{l}{s^l_0} \cap \bigcap_{l=k-1}^{n-1+k-1} \Cyl{l}{s^{k-1}_l},
	\end{equation*}
that is, a generic element of $\parti_n^k$ is a cylinder of $\mathcal{C}_{-n, \ldots, n+k-2}$, and every such cylinder can be formed in this way because the symbols $s^0_{-n}, \ldots, s^0_0, \ldots, s^{k-1}_0, \ldots, s^{k-1}_{n-1}$ can be chosen arbitrarily, so we conclude that $\parti_n^k = \mathcal{C}_{-n, \ldots, n+k-2}$.
\end{proof}

It is now trivial to conclude the following last results.

\begin{lemma}
\label{prop:measure_entropy_of_Pn}
$\enm(\shift_{\phi}, \parti_n) = \Enm(\mathcal C_0)$.
\end{lemma}
\begin{proof}
From \cref{prop:correfinamento.cilindros,prop:entropia.cilindros} it follows that
	\begin{equation*}
	\Enm(\parti_n^k) = \Enm(\mathcal C_{-n, \ldots, n-1+k-1}) = n \Enm(\mathcal C_{-1}) + (n+k-1) \Enm(\mathcal C_0),
	\end{equation*}
therefore
	\begin{equation*}
	\begin{split}
	\enm(\shift_{\phi}, \parti_n) &= \lim_{k \to \infty} \frac{1}{k} \Enm(\parti_n^k) \\
		&= \lim_{k \to \infty} \frac{1}{k} (n \Enm(\mathcal C_{-1}) + (n+k-1) \Enm(\mathcal C_0)) \\
		&= \Enm(\mathcal C_0).
		\qedhere
	\end{split}
	\end{equation*}
\end{proof}

\begin{proof}[Proof of {\cref{thm:measure_entropy_of_extended_shift}}]
\phantomsection
\addcontentsline{toc}{subsection}{Proof of \texorpdfstring{\cref{thm:measure_entropy_of_extended_shift}}{Theorem~A}}
The sequence of partitions $\parti_n = \mathcal C_{-n, \ldots, n-1}$ ($n \in \N$) is increasing relative to the refinement order:
	\begin{equation*}
	\parti_0 \preceq \parti_1 \preceq \cdots \preceq \parti_n \preceq \cdots.
	\end{equation*}
Besides that, the union of $\parti_n$ generates the $\sigma$-algebra of the space $\Shift_{m_-,m_+}$. Finally, the entropy of $\parti_n$ is finite, because the entropy of $\mathcal C_{-1}$ and $\mathcal C_0$ are finite. Therefore, by the Kolmogorov-Sinai theorem (\cref{prop:Kolmogorov-SinaiTheorem}), the metric entropy of the system is
	\begin{equation*}
	\enm(\shift_{\phi}) = \lim_{n \to \infty} \enm(\shift_{\phi}, \parti_n).
	\end{equation*}
We thus have to calculate $\enm(\shift_{\phi}, \parti_n)$, which is, by definition,
	\begin{equation*}
	 \enm(\shift_{\phi}, \parti_n) := \lim_{k \to \infty} \frac{1}{k} \Enm(\parti_n^k),
	\end{equation*}
which shows we have to calculate $\Enm(\parti_n^k)$.

This finally implies that
	\begin{equation*}
	\enm(\shift_{\phi}) = \lim_{n \to \infty} \enm(\shift_{\phi}, \parti_n) = \Enm(\mathcal C_0),
	\end{equation*}
which by \cref{prop:entropy_basic_cilinder_partitions} equals $\sum_{s=0}^{m_+-1} -p^+_{s} \log p^+_{s}$.
\end{proof}

\section{Folding entropy of extended shifts}
\label{sec:proof_folding_entropy_formula}

Let $(\Shift_{m_-,m_+}, \shift_{\phi})$ be an extended shift space. As a consequence of $\phi$ being surjective, we have that $m_+ \geq m_-$. When $m_+ > m_-$, the extended shift $\shift_{\phi}$ is not invertible and, for any given $x \in \Shift_{m_-, m_+}$, the set $\shift_{\phi}\inv(x)$ has more than one element. In the following discussion, we will need a way the refer to each element of $\shift_{\phi}\inv(x)$, so, for each $s \in \phi\inv(x_{-1})$, we define
\begin{equation}
	\label{eq:inverse_image_of_extended_sequence}
	\hat x(s) := (\ldots, x_{-2}; s, x_0, \ldots).
\end{equation}
We also denote
\begin{equation}
	\label{eq:inverse_imeage_extended_set}
	\hat x := \shift_{\phi}\inv(x) = \set{\hat x(s)}{s \in \phi\inv(x_{-1})}
\end{equation}	
and, for each $X \subseteq \Shift_{m_-,m_+}$,
\begin{equation*}
	\hat X := \set{\hat x}{x \in X} \subseteq \shift_{\phi}\inv(\epsilon).
\end{equation*}

From \cref{def:folding_entropy}, the folding entropy of $\shift_{\phi}$ is given by
\begin{equation*}
	\mathcal F(\shift_{\phi}) = \Enm(\epsilon \mid \shift_{\phi}\inv(\epsilon))
\end{equation*}
and, from \cref{eq:conditional_entropy_using_disintegration}, the conditional entropy of the atomic partition $\epsilon$ with respect to the dynamical pullback $\shift_{\phi}\inv(\epsilon) = \set{\hat x}{x \in \Shift_{m_-,m_+}}$ can be calculated by
    \begin{equation*}
	\Enm(\epsilon \mid \shift_{\phi}\inv(\epsilon)) = \int_{\hat x \in \shift_{\phi}\inv(\epsilon)} \Enm[\med_{\hat x}](\epsilon|_{\hat x}) \hat \med(\dd \hat x),
    \end{equation*}
in which $(\med_{\hat x})_{\hat x \in \shift_{\phi}\inv(\epsilon)}$ is the disintegration of $\med$ with respect to $\shift_{\phi}\inv(\epsilon)$ and $\hat\med$ is the quotient measure of $\shift_{\phi}\inv(\epsilon)$.

So in order to calculate the folding entropy of the extended shift, we need to find the quotient measure $\hat\med$ and to disintegrate the measure $\med$ with respect to the dynamical pullback $\shift_{\phi}\inv(\epsilon)$ of the atomic partition $\epsilon$ of $\Shift_{m_-,m_+}$ (defined in \cref{sec:conditional_entropy_and_folding_entropy}).

\subsection{The quotient measure}
\label{ssec:quotient_measure}

Let us denote the natural projection with respect to the partition $\shift_{\phi}\inv(\epsilon)$ by $\fun{\proj}{\Shift_{m_-,m_+}}{\shift_{\phi}\inv(\epsilon)}$.
Let us first determine the quotient $\sigma$-algebra $\hat{\mathcal{B}}$, which is the pushforward of the cylinders $\sigma$-algebra of $\Shift_{m_-,m_+}$ by the natural projection $\proj$.

\begin{proposition}
	\label{prop:extended_shift_pullback_projected_set}
	For every set $X \subseteq \Shift_{m_-,m_+}$,
	\begin{equation*}
		\proj\inv(\hat X) = \shift_{\phi}\inv(X).
	\end{equation*}
	Besides that, the quotient $\sigma$-algebra $\hat{\mathcal{B}}$ is generated by the projected cylinder sets $\hat C$ ($C \in \mathcal{B}$ is a cylinder).
\end{proposition}
\begin{proof}
	The first claim follows directly from
	\begin{equation*}
		\proj\inv (\hat X) = \bigcup \hat X = \bigcup \set{\shift_{\phi}\inv(x)}{x \in X} = \shift_{\phi}\inv(X).
	\end{equation*}
	
	Now that $\mathcal{Q} \subseteq \shift_{\phi}\inv(\epsilon)$. Since each element of $\shift_{\phi}\inv(\epsilon)$ is of the form $\hat x$ for some $x \in \Shift_{m_-, m_+}$, there exists a set $X \subseteq \Shift_{m_-, m_+}$ such that $\mathcal{Q} = \set{\hat x}{x \in X} = \hat X$. This implies that its inverse image by the projection is of the form $\proj\inv(\mathcal{Q}) = \proj\inv(\hat X) = \shift_{\phi}\inv(X)$. This shows that $\hat{\mathcal{B}}$ is generated by sets $\hat C$ such that  $\shift_{\phi}\inv(C) \in \mathcal{B}$ is a cylinder, which means that $C$ is also a cylinder.
\end{proof}

Since $\hat x = \shift_{\phi}\inv(x)$, it may be confusing to understand the difference between the sets $\hat C$ and $\shift_{\phi}\inv(C)$. To better understand the notation, it is worth noticing that, if $x \in C$, then $\hat x = \shift_{\phi}\inv(x) \subseteq \shift_{\phi}\inv(C)$; that is, for each $s \in \phi\in(x_{-1})$, we have $\hat x(s) \in \shift_{\phi}\inv(C)$. This shows that the elements of the set $\hat x$ (which is an element of $\hat C$) do not belong to the set $\hat C$, but instead to $\shift_{\phi}\inv(C)$. To further avoid confusion, consider this example. Suppose $x, y \in \Shift_{m_-, m_+}$, $\hat x = \{\hat x(0), \hat x(1)\}$ and $\hat y = \{\hat y(0), \hat y(1)\}$. If $C = \{x, y\}$, then
\begin{equation*}
	\hat C = \{\hat x, \hat y\} = \{\{\hat x(0), \hat x(1)\}, \{\hat y(0), \hat y(1)\}\},
\end{equation*}
while $\shift_{\phi}\inv(C) = \{\hat x(0), \hat x(1), \hat y(0), \hat y(1)\}$.

In particular, it is worth noting that, for a cylinder $\Cyl{i}{s}$,
\begin{equation*}
	\proj\inv(\hatCyl{i}{s}) = \shift_{\phi}\inv(\Cyl{i}{s}) =
	\begin{cases}
		\Cyl{i+1}{s} & i \neq -1 \\
		\bigcup_{s' \in \phi\inv(s)} \Cyl{0}{s'} & i = -1.
	\end{cases}
\end{equation*}

The quotient measure $\hat \med := \proj\emp\med$ on $\shift_{\phi}\inv(\epsilon)$ is the pushforward of $\med$ by the natural projection $\fun{\proj}{\Shift_{m_-, m_+}}{\shift_{\phi}\inv(\epsilon)}$ of the dynamical pullback of the atomic partition. The next proposition shows how we can easily calculate it using the original measure $\med$.

\begin{proposition}
	\label{prop:extended_shift_quotient_measure}
	Let $(\Shift_{m_-,m_+}, \shift_{\phi})$ be an extended shift space. For every measurable set $M \subseteq \Shift_{m_-, m_+}$,
	\begin{equation*}
		\hat\med(\hat M) = \med(M).
	\end{equation*}
\end{proposition}
\begin{proof}
	Since $\proj\inv(\hat M) = \shift_{\phi}\inv(M)$ (\cref{prop:extended_shift_pullback_projected_set}) and $\shift_{\phi}$ is measure-preserving (\cref{prop:extended_shift_preserves_measure}), it follows that
	\begin{equation*}
		\hat \med(\hat M) = \med(\proj\inv(\hat M)) = \med(\shift_{\phi}\inv(M)) = \med(M).
		\qedhere
	\end{equation*}
\end{proof}

\subsection{Disintegration and calculating the folding entropy}
\label{ssec:disintegration}

We wish to disintegrate the measure $\med$ on $\Shift_{m_-,m_+}$ with respect to the pullback partition $\shift_{\phi}\inv(\epsilon)$. In order to do that, we must find, for each $\hat x \in \shift_{\phi}\inv(\epsilon)$, the conditional measures $\med_{\hat x}$ on $\Shift_{m_-,m_+}$, in such a way that, for every measurable set $M \in \mathcal B$, it holds that
\begin{equation*}
	\med(M) = \int_{\hat x \in \shift_{\phi}\inv(\epsilon)} \med_{\hat x}(M) \hat \med (\dd \hat x).
\end{equation*}

To define the conditional measures on $\hat x$, remember that $\hat x = \set{\hat x(s)}{s \in \phi\inv(x_{-1})}$ and that the conditional measure is supported on $\hat x$, so, for each measurable set $M \in \mathcal B$, it is given by $\med_{\hat x}(M) = \med_{\hat x}(M \cap \hat x)$. Thus, since $\hat x$ is finite, we can define it on each atom $\{\hat x(s)\}$.

Based on the probability distribution $p^+$ on $\{0, \ldots, m_+-1\}$, we have described how to induce a probability distribution $p^-$ on $\{0, \ldots, m_--1\}$ by taking the pushforward of $p^+$ by the transition function $\phi$. Using the two measures $p^+$ on $\{0, \ldots, m_+-1\}$ and $p^-$ on $\{0, \ldots, m_--1\}$, we can define, for each $s_- \in \{0, \ldots, m_--1\}$, a new probability measure $q^{s_-}$ on the inverse image set $\phi\inv(s_-)$ by setting, for each $s_+ \in \phi\inv(s_-)$
    \begin{equation*}
	q^{s_-}_{s_+} := \frac{p^+_{s_+}}{p^-_{s_-}}.
    \end{equation*}
This is a probability measure because, for each $s_- \in \{0, \ldots, m_--1\}$,
    \begin{equation*}
	\sum_{s_+ \in \phi\inv(s_-)} q^{s_-}_{s_+} = \sum_{s_+ \in \phi\inv(s_-)} \frac{p^+_{s_+}}{p^-_{s_-}} = \frac{\sum_{s_+ \in \phi\inv(s_-)}  p^+_{s_+}}{p^-_{s_-}} =  1.
    \end{equation*}
It is important to notice that, as a direct consequence of this definition,
\begin{equation}
\label{eq:positive_distribution_is_negative_distribution_times_conditional_distribution}
	p^+ = (p^+_{s_+})_{s_+ = 0}^{m_+-1} = ((p^-_{s_-} q^{s_-}_{s_+})_{s_+ \in \phi\inv(s_-)})_{s_- = 0}^{m_--1}.
\end{equation}

We use these measures $q^{s_-}$ to define the conditional measures as follows, by identifying the set $\hat x$ with the preimage $\phi\inv(x_{-1})$.

\begin{definition}
	\label{def:extended_shift_conditional_measure}
	Let $(\Shift_{m_-,m_+}, \shift_{\phi})$ be an extended shift space with measure $\med$ given by the probability distribution $p^+$, and let $\hat x \in \shift_{\phi}\inv(\epsilon)$. The \emph{conditional measure $\med_{\hat x}$ on $\hat x$} is the probability measure defined, for each $s \in \phi\inv(x_{-1})$, by
	\begin{equation*}
		\med_{\hat x} (\{\hat x(s)\}) := q^{x_{-1}}_{s} = \frac{p^+_{s}}{p^-_{x_{-1}}}.
	\end{equation*} 
\end{definition}

Now we show this is the disintegration of $\med$.

\begin{proposition}
	\label{prop:extended_shift_disintegration}
	Let $(\Shift_{m_-,m_+}, \shift_{\phi})$ be an extended shift space. The family $\{\med_{\hat x}\}_{\hat x \in \shift_{\phi}\inv(\epsilon)}$ is the disintegration of $\med$ with respect to $\shift_{\phi}\inv(\epsilon)$.
\end{proposition}
\begin{proof}
	It suffices to show that, for each basic cylinder $\Cyl{s}{i}$, it holds that
	\begin{equation*}
		\med(\Cyl{s}{i}) = \int_{\hat x \in \shift_{\phi}\inv(\epsilon)} \med_{\hat x}(\Cyl{s}{i} \cap \hat x) \hat \med (\dd \hat x).
	\end{equation*}
	
	First let us calculate the sets $\Cyl{s}{i} \cap \hat x$. For any set $C \subseteq \Shift_{m_-, m_+}$, it holds that $x \in C$ if, and only if, $\hat x \subseteq \shift_{\phi}\inv(C)$. Because of this, we must consider the cases $\hat x \in \widehat{\shift_{\phi}(\Cyl{s}{i})}$ and $\hat x \notin \widehat{\shift_{\phi}(\Cyl{s}{i})}$; or equivalently, $x \in \shift_{\phi}(\Cyl{s}{i})$ and $x \notin \shift_{\phi}(\Cyl{s}{i})$. According to \cref{prop:extended_shift_iterated_cylinders}, the expression for $\shift_{\phi}(\Cyl{s}{i})$ depends on the value for $i$, so we consider $2$ scenarios:
	
	\begin{itemize}
		\item ($i \neq 0$) In this case, we have $\shift_{\phi}(\Cyl{s}{i}) = \Cyl{s}{i-1}$, hence
		\begin{equation*}
			\Cyl{s}{i} \cap \hat x =
			\begin{cases}
				\hat x		& \hat x \in \hatCyl{s}{i-1} \\
				\emptyset	& \hat x \notin \hatCyl{s}{i-1}.
			\end{cases}
		\end{equation*}
		Since $\med_{\hat x}(\hat x) = 1$ e $\med_{\hat x}(\emptyset) = 0$, it follows that
		\begin{align*}
			\med(\Cyl{i}{s}) &= \med(\Cyl{i-1}{s}) \\
			&= \hat \med(\hatCyl{i-1}{s}) \\
			&= \int_{\hat x \in \hatCyl{i-1}{s}} 1 \hat \med (\dd \hat x) + \int_{\hat x \in \shift_{\phi}\inv(\epsilon) \setminus \hatCyl{i-1}{s}} 0 \hat \med (\dd \hat x) \\
			&= \int_{\hat x \in \hatCyl{i-1}{s}} \med_{\hat x}(\hat x) \hat \med (\dd \hat x) + \int_{\hat x \in \shift_{\phi}\inv(\epsilon) \setminus \hatCyl{i-1}{s}} \med_{\hat x}(\emptyset) \hat \med (\dd \hat x) \\
			&= \int_{\hat x \in \shift_{\phi}\inv(\epsilon)} \med_{\hat x}(\Cyl{i}{s} \cap \hat x) \hat \med (\dd \hat x).
		\end{align*}
		
		\item ($i = 0$) In this case, we have that
		\begin{equation*}
			\Cyl{0}{s} \cap \hat x =
			\begin{cases}
				\{\hat x(s)\}	& \hat x \in \hatCyl{-1}{\phi(s)} \\
				\emptyset		& \hat x \notin \hatCyl{-1}{\phi(s)}.
			\end{cases}
		\end{equation*}
		Since $\med_{\hat x}(\{\hat x(s)\}) = q^{x_{-1}}_{s}$ and $\med_{\hat x}(\emptyset) = 0$ (and, for each $x \in \Cyl{-1}{\phi(s)}$, it holds that $x_{-1} = \phi(s)$), it follows that
		\begin{align*}
			\med(\Cyl{0}{s}) &= q^{\phi(s)}_{s} \med(\Cyl{-1}{\phi(s)}) \\
			&= q^{\phi(s)}_{s} \hat \med(\hatCyl{-1}{\phi(s)}) \\
			&= \int_{\hat x \in \hatCyl{-1}{\phi(s)}} q^{x_{-1}}_{s} \hat \med (\dd \hat x) + \int_{\hat x \in \shift_{\phi}\inv(\epsilon) \setminus \hatCyl{-1}{\phi(s)}} 0 \hat \med (\dd \hat x) \\
			&= \int_{\hat x \in \hatCyl{-1}{\phi(s)}} \med_{\hat x}(\Cyl{0}{s} \cap \hat x) \hat \med (\dd \hat x) + \int_{\hat x \in \shift_{\phi}\inv(\epsilon) \setminus \hatCyl{-1}{\phi(s)}} \med_{\hat x}(\Cyl{0}{s} \cap \hat x) \hat \med (\dd \hat x) \\
			&= \int_{\hat x \in \shift_{\phi}\inv(\epsilon)} \med_{\hat x}(\Cyl{0}{s} \cap \hat x) \hat \med (\dd \hat x).
			\qedhere
		\end{align*}
	\end{itemize}
\end{proof}

We are finally ready to prove our main result on the folding entropy.

\begin{proof}[Proof of {\cref{thm:folding_entropy_formula}}]
\phantomsection
\addcontentsline{toc}{subsection}{Proof of \texorpdfstring{\cref{thm:folding_entropy_formula}}{Theorem~B}}
	As discussed in the beginning of the section, it follows from \cref{def:folding_entropy} and \cref{eq:conditional_entropy_using_disintegration} that the folding entropy of $\shift_{\phi}$ is given by
	\begin{equation*}
		\mathcal F(\shift_{\phi}) = \int_{\hat x \in \shift_{\phi}\inv(\epsilon)} \Enm[\med_{\hat x}](\epsilon|_{\hat x}) \hat \med(\dd \hat x),
	\end{equation*}
	in which $\hat \med = \med_{\shift_{\phi}\inv(\epsilon)}$ is the quotient measure of $\shift_{\phi}\inv(\epsilon)$.
	
	Now notice that
	\begin{equation*}
		\epsilon|_{\hat x} = \set{\{y\} \cap \hat x}{\{y\} \in \epsilon} = \set{\{\hat x(s_+)\}}{s_+ \in \phi\inv(x_{-1})},
	\end{equation*}
	hence from \cref{eq:entropy_of_countable_partition} and \cref{def:extended_shift_conditional_measure} it follows that
	\begin{equation*}
		\Enm[\med_{\hat x}](\epsilon|_{\hat x}) = \smashoperator{\sum_{s_+ \in \phi\inv(x_{-1})}} -\med_{\hat x} (\{\hat x(s_+)\}) \log \med_{\hat x} (\{\hat x(s_+)\}) = \smashoperator{\sum_{s_+ \in \phi\inv(x_{-1})}} -q^{x_{-1}}_{s_+} \log q^{x_{-1}}_{s_+}.
	\end{equation*}
	This shows that this value depends only on $x_{-1}$, so it is constant on each set $\hatCyl{-1}{s_-}$.
	The set
	\begin{equation*}
		\hat {\mathcal C}_{-1} := \set{\hatCyl{-1}{s_-}}{s_- \in \{0, \ldots, m_--1\}}
	\end{equation*}
	is a partition of $\shift_{\phi}\inv(\epsilon)$, since (1) $\hatCyl{-1}{s_-} \neq \emptyset$; (2) $\hatCyl{-1}{s_-} \cap \hatCyl{-1}{r^-} = \emptyset$ when $s_- \neq r^-$; and (3) $\shift_{\phi}\inv(\epsilon) =\allowbreak \bigcup_{s_- = 0}^{m_--1} \hatCyl{-1}{s_-}$.
	
	Besides that, it follows from \cref{prop:extended_shift_quotient_measure} and \cref{def:extended_shift_measure} that $\hat\med (\hatCyl{-1}{s_-}) = \med (\Cyl{-1}{s_-}) = p^-_{s_-}$. Thus the folding entropy of $\shift_{\phi}$ is
	\begin{equation*}
		\begin{split}
			\mathcal F(\shift_{\phi})
			&= \int_{\hat x \in \shift_{\phi}\inv(\epsilon)} \Enm[\med_{\hat x}](\epsilon|_{\hat x}) \hat \med (\dd \hat x) \\
			&= \sum_{s_- = 0}^{m_--1} \int_{\hat x \in \hatCyl{-1}{s_-}} \Enm[\med_{\hat x}](\epsilon|_{\hat x}) \hat\med (\dd \hat x) \\
			&= \sum_{s_- = 0}^{m_--1} \Big( \smashoperator[r]{\sum_{s_+ \in \phi\inv(s_-)}} -q^{s_-}_{s_+} \log q^{s_-}_{s_+} \Big) \hat\med (\hatCyl{-1}{s_-}) \\
			&= \sum_{s_- = 0}^{m_--1} \Big(\smashoperator[r]{\sum_{s_+ \in \phi\inv(s_-)}} -q^{s_-}_{s_+} \log q^{s_-}_{s_+} \Big) p^-_{s_-}.
		\end{split}
	\end{equation*}
	
	Noting that $q^{s_-}_{s_+} p^-_{s_-} = p^+_{s_+}$ (\cref{def:extended_shift_conditional_measure}) and $p^-_{s_-} = \sum_{s_+ \in \phi\inv(s_-)} p^+_{s_+}$, it follows that
	\begin{equation*}
		\begin{split}
			\mathcal{F}(\shift_{\phi})
			&= \sum_{s_- = 0}^{m_--1} \smashoperator[r]{\sum_{s_+ \in \phi\inv(s_-)}} -q^{s_-}_{s_+}p^-_{s_-} \log q^{s_-}_{s_+}  \\
			&= \sum_{s_- = 0}^{m_--1} \smashoperator[r]{\sum_{s_+ \in \phi\inv(s_-)}} -p^+_{s_+} (\log p^+_{s_+} - \log p^-_{s_-}) \\
			&= \sum_{s_+ = 0}^{m_+-1} -p^+_{s_+} \log p^+_{s_+} - \sum_{s_- = 0}^{m_--1} -\left( \smashoperator[r]{\sum_{s_+ \in \phi\inv(s_-)}} p^+_{s_+} \right) \log p^-_{s_-} \\
			&= \sum_{s_+ = 0}^{m_+-1} -p^+_{s_+} \log p^+_{s_+} - \sum_{s_- = 0}^{m_-1} -p^-_{s_-} \log p^-_{s_-}.
		\end{split}
	\end{equation*}
	
 	Finally, since by \cref{prop:entropy_basic_cilinder_partitions} we have $\Enm(\mathcal C_0) = \sum_{s_+ = 0}^{m_+-1} -p^+_{s_+} \log p^+_{s_+}$ and $\Enm(\mathcal C_{-1}) = \sum_{s_- = 0}^{m_--1} -p^-_{s_-} \log p^-_{s_-}$, we conclude that
	\begin{equation*}
		\mathcal{F}(\shift_{\phi}) = \Enm(\mathcal C_0) - \Enm(\mathcal C_{-1}).
		\qedhere
	\end{equation*}
\end{proof}

In particular, since the metric entropy is given by $\enm(\shift_{\phi}) = \Enm(\mathcal C_0)$, then
    \begin{equation*}
	\enm(\shift_{\phi}) = \mathcal{F}(\shift_{\phi}) + \Enm(\mathcal C_{-1}).
    \end{equation*}

This can be interpreted as showing that the metric entropy is the sum of the entropy of the backwards alphabet with a component of non-invertibility, the folding entropy. In the trivial case $m_+ = m_-$, the folding entropy is zero, which shows that the metric entropy is equal to the entropy of the backward alphabet, which in this case equals the forward alphabet.

\notocsection{Acknowledgment}
The authors would like to thank Pouya Mehdipour for her comments.
N. M. was partially financed by the Coordenação de Aperfeiçoamento de Pessoal de Nível Superior
Brasil (CAPES) - grant 88887.645688/2021-00.
P. M. was partially financed by the Coordenação de Aperfeiçoamento de Pessoal de Nível Superior
Brasil (CAPES) - grant 141401/2020-6.
R.V. was partially supported by Conselho Nacional de Desenvolvimento Científico e Tecnológico (CNPq) (grants 313947/2020-1 and 314978/2023-2), and partially supported by Fundação de Amparo à Pesquisa do Estado de São Paulo (FAPESP) (grants 17/06463-3 and 18/13481-0). 

\printbibliography

\end{document}